\documentclass[10pt]{amsart}

\usepackage{amssymb,amsmath,amsthm}
 
\usepackage{hyperref}
\usepackage{color}
\usepackage{mathrsfs}

\makeatletter
\def\blfootnote{\gdef\@thefnmark{}\@footnotetext}
\makeatother

\renewcommand\mathcal{\mathscr}

\usepackage{ulem}
\renewcommand{\emph}{\normalem}

\theoremstyle{plain}
\newtheorem{theorem}{Theorem}[section]
\newtheorem*{theorem*}{Theorem}
\newtheorem{lemma}[theorem]{Lemma}
\newtheorem*{lemma*}{Lemma}
\newtheorem{corollary}[theorem]{Corollary}
\newtheorem{proposition}[theorem]{Proposition}

\theoremstyle{remark}
\newtheorem{remark}[theorem]{Remark}
\newtheorem*{remark*}{Remark}

\theoremstyle{definition}
\newtheorem{definition}[theorem]{Definition}
\newtheorem*{definition*}{Definition}

\theoremstyle{notation}%%%%%%%%%%%

\newtheorem*{notation*}{Notation}

\newtheorem{Basic assumptions}[theorem]{Basic assumptions}

\numberwithin{equation}{section}

\newcount\quantno
\everydisplay{\quantno=0}\everycr{\quantno=0}
\newcommand\quant{\advance\quantno by1
                      \ifnum\quantno=1\qquad\else\quad\fi\forall }
\newcommand\itemno[1]{(\romannumeral #1)}

\newcommand\Dom{\mathrm{Dom}}
\newcommand\Ran{\mathrm{Ran}}

\newcommand\rest[1]{\kern-.1em
          \lower.5ex\hbox{$\scriptstyle #1$}\kern.05em}

\newcommand\set[1]{{\left\{#1\right\}}}

\renewcommand\mod[1]{\vert{#1}\vert}
\newcommand\bigmod[1]{\bigl\vert{#1}\bigr|}
\newcommand\Bigmod[1]{\Bigl\vert{#1}\Bigr|}

\newcommand\norm[2]{{\Vert{#1}\Vert_{#2}}}

\newcommand\bignorm[2]{\left.{\bigl\Vert{#1}\bigr\Vert_{#2}}\right.}

\newcommand\bigopnorm[2]{\bigl|\!\bigl|\!\bigl| {#1} 
\bigr|\!\bigr|\!\bigr|_{#2}}

\newcommand\smallfrac[2]{\mbox{\small$\displaystyle\frac{#1}{#2}$}}
\newcommand\wrt{\,\text{\rm d}}

\newcommand\BR{\mathbb{R}}

\newcommand\cB{\mathcal{B}}   
 
\newcommand\cC{\mathcal{C}}

\newcommand\cH{\mathcal{H}}   
 
\newcommand\cI{\mathcal{I}}

\newcommand\cL{\mathcal{L}}

\newcommand\al{\alpha}
\newcommand\be{\beta}
    \newcommand\Ga{\Gamma}

  \newcommand\vep{\varepsilon}

\newcommand\la{\lambda}   
    \newcommand\Om{\Omega}

\newcommand\vp{\varphi}

\newcommand\OV{\overline}
\newcommand\funnyk{k\hbox to 0pt{\hss\phantom{g}}}

\newcommand\lu[1]{L^1(#1)}
\newcommand\lp[1]{L^p(#1)}

\newcommand\ld[1]{L^2(#1)}

\newcommand\ly[1]{L^\infty(#1)}

\newcommand\Xu[1]{X^1(#1)}

\newcommand\Xh[1]{X^k(#1)}

\newcommand\whH{\widehat{\phantom{G}}\hbox to 0pt{\hss $H$}}

\newcommand\emspace{\hbox to 6pt{\hss}}
\newcommand\ds{\displaystyle}

\newcommand\rmi{\hbox{\rm (i)}}
\newcommand\rmii{\hbox{\rm (ii)}}
\newcommand\rmiii{\hbox{\rm (iii)}}

\newcommand\ioty{\int_0^{\infty}}
\newcommand\dtt[1]{\,\frac{\mathrm {d} #1}{ #1}}

\newcommand\One{{\mathbf{1}}}

\newcommand\e{\mathrm{e}}

\newcommand\Inj{\mathrm{Inj}}

\newcommand\dest{\text{\rm d}}

\newcommand\Ric{\mathop{\rm Ric}}

\newcommand\supp{\mathrm{supp}}

\DeclareSymbolFont{EUEX}{U}{euex}{m}{n}

\DeclareSymbolFont{euexlargesymbols}{U}{euex}{m}{n}
\DeclareMathSymbol{\intop}{\mathop}{euexlargesymbols}{"52}
     \def\int{\intop\nolimits}

\DeclareSymbolFont{euexsymbols}     {U}{euex}{m}{n}
\DeclareMathSymbol{\smallint}{\mathop}{euexsymbols}{"52}

\begin{document}

\title[Sharp boundedness]
{Sharp endpoint results \\
for imaginary powers and Riesz transforms \\
on certain noncompact manifolds}

\subjclass[2000]{} 


\thanks{Work partially supported by PRIN 2010 ``Real and complex manifolds: 
geometry, topology and harmonic analysis".}

\author[G. Mauceri, S. Meda and M. Vallarino]
{Giancarlo Mauceri, Stefano Meda and Maria Vallarino}

\address{Giancarlo Mauceri: Dipartimento di Matematica\\ 
Universit\`a di Genova\\
via Dodecaneso 35\\ 16146 Genova\\ Italy 
-- mauceri@dima.unige.it}

\address{Stefano Meda: 
Dipartimento di Matematica e Applicazioni
\\ Universit\`a di Milano-Bicocca\\
via R.~Cozzi 53\\ I-20125 Milano\\ Italy
-- stefano.meda@unimib.it}

\address{Maria Vallarino:
Dipartimento di Scienze Matematiche
\\ Politecnico di Torino\\
corso Duca degli Abruzzi 24\\ 10129 Torino\\ Italy
 --  maria.vallarino@polito.it}

\begin{abstract}
In this paper we consider a complete connected noncompact 
Riemannian manifold $M$ with bounded geometry and spectral gap.  
We prove that the imaginary powers of the Laplacian and the Riesz transform are bounded from the Hardy space $X^1(M)$, introduced in previous work of the authors, to $L^1(M)$.  
\end{abstract}
%VERSIONE DI \today
\maketitle

%\tableofcontents

\setcounter{section}{0}
\section{Introduction} \label{s:Introduction}

Denote by $M$ a complete connected noncompact Riemannian manifold
of dimension $n$
with Ricci curvature bounded from below, positive injectivity radius 
and spectral gap.  Denote by $\cL$ (minus) the Laplace--Beltrami operator
on $M$.  Denote by $\Xh{M}$ the Hardy-type spaces introduced in 
\cite{MMV1,MMV2}.  The purpose of this paper is to prove the following result.

\begin{theorem} \label{t: main}
For every $u$ in $\BR$ the operators $\cL^{iu}$ and $\nabla\cL^{-1/2}$ 
are bounded from $\Xu{M}$ to $\lu{M}$.
\end{theorem}

In \cite{MMV1,MMV2} we proved that the operators $\cL^{iu}$ and $\nabla\cL^{-1/2}$
are bounded from $\Xh{M}$ to $\lu{M}$ for an integer $k$ large enough
and depending on $n$.  Clearly Theorem~\ref{t: main} is an improvement
of the aforementioned results.  We believe that its main interest lies not
only in the fact that all these operators are bounded from the
same space $\Xu{M}$ to $\lu{M}$, but also in the method of proof,
which appear to be quite adaptable to the geometry of manifolds and could pave the way to obtaining similar results for more general manifolds.
%We hope that this will open the way to further investigations.

The imaginary powers of $\cL$ and the Riesz transforms on Riemannian manifolds have been investigated in a number of papers \cite{A1, A2, ACDH, AMR, CMM1,CMM2, CD, DY, HLMMY, I, MRu, MMV1, MMV2, MV, Ru, T}. 
For a discussion of these papers and their  relations to our results we refer the reader to the introductions of \cite{MMV1, MMV2}.
\par
We now give a brief outline of the paper. In Section \ref{s: Background material} we recall the definition and the basic properties of the atomic Hardy space $X^1(M)$. In Section \ref{s: Atoms and the LB} we estimate the $L^2$ norm of the resolvent  of the Laplacian $\cL$ on atoms. In Section \ref{s: IP} we prove the boundedness of the imaginary powers of $\cL$ and in Section \ref{s: RT exponential case} that of the Riesz transform $\nabla\cL^{-1/2}$.  In the last section we briefly indicate how the arguments of the previous sections may be adapted to doubling manifolds that satisfy Gaussian upper estimates.
\par
We shall use the ``variable constant convention'', and denote by $C,$
possibly with sub- or superscripts, a constant that may vary from place to 
place and may depend on any factor quantified (implicitly or explicitly) 
before its occurrence, but not on factors quantified afterwards.

\section{Background on Hardy-type spaces}
\label{s: Background material}

Let $M$ denote a connected, complete $n$-dimensional Riemannian manifold
of infinite volume with Riemannian measure $\mu$.  
Denote by $\Ric$ the Ricci tensor, by $-\cL$ the Laplace--Beltrami operator
on $M$, 
by $b$ the bottom of the $\ld{M}$ spectrum of $\cL$,
and set $\be =
\limsup_{r\to\infty} \bigl[\log\mu\bigl(B(o,r)\bigr)\bigr]/(2r)$, 
where $o$ is any reference point of $M$.
By a result of Brooks $b\leq \be^2$ \cite{Br}.
\par
We denote by $\cB$ the family of all geodesic balls on $M$.
For each $B$ in $\cB$ we denote by $c_B$ and $r_B$
the centre and the radius of $B$ respectively.  
Furthermore, we denote by $c \, B$ the
ball with centre $c_B$ and radius $c \, r_B$.
For each \emph{scale parameter} $s$ in $\BR^+$, 
we denote by $\cB_s$ the family of all
balls $B$ in $\cB$ such that $r_B \leq s$.  
\par
We assume that the 
injectivity radius of $M$ is positive,
that the Ricci tensor is bounded from below
 and that $M$ has spectral gap,
to wit $b>0$. 
It is well known that for manifolds satisfying 
the assumptions above
there are positive constants $\al$, $\be$ and $C$ such that
 %for all balls $B$ of radius $r_B\ge 1$
\begin{equation} \label{f: volume growth} 
\mu(B)
\leq C \, r_B^{\al} \, \e^{2\be \, r_B}
\quant B \in\cB,\ \ \textrm{such \ that\ \ } r_B\ge 1.
\end{equation}
%where $B(p,r)$ denotes the  
%geodesic ball with centre $p$ and radius~$r$. 
Moreover, the measure $\mu$ is \emph{locally doubling}, 
i.e. for every $s>0$ there exists a constant $D_s$ such that 
$$
\mu(2B)\le D_s\ \mu(B) \qquad
\forall B\in \cB_s.
$$
Furthermore (see \cite[Remark~2.3]{MMV2}) there exists a positive 
constant $C$ such that
\begin{equation} \label{f: lower bound balls}
C^{-1}\,r_B^n
\leq \mu(B)
\leq C\,r_B^n
\quant B\in\cB_1.
\end{equation}
%We denote by $\cB$ the family of all balls on $M$.
%For each $B$ in $\cB$ we denote by $c_B$ and $r_B$
%the centre and the radius of $B$ respectively.  
%Furthermore, we denote by $c \, B$ the
%ball with centre $c_B$ and radius $c \, r_B$.
%For each \emph{scale parameter} $s$ in $\BR^+$, 
%we denote by $\cB_s$ the family of all
%balls $B$ in $\cB$ such that $r_B \leq s$.  

In this section we gather some known facts about 
the Hardy-type space $\Xu{M}$, introduced
in \cite{MMV1} and studied in \cite{MMV2, MMV3}.  
For each open ball $B$, we denote by \begin{itemize}
\item[\rmi] $h^2(B)$ the space of all $\cL$-harmonic functions  in $L^2(B)$;
\item[\rmii] $q^2(B)$ the space all  functions $u\in L^2(B)$ such that   $\cL u$ is constant on $B$.
\end{itemize}
We say that a function $u$ lies in the space $h^2(\overline{B})$ (respectively $q^2(\overline{B})$) if $u$ is the restriction to $\overline{B}$ of a function in $h^2({B'})$ 
(respectively $q^2({B'})$) for some open ball $B'$ containing $B$.

%\item[\rmiii] $h^2(\overline{B})$ the space of all functions that are $\cL$-harmonic in a neighbourhood of $\overline{B}$
%\item[\rmiv]  $q^2(\overline{B})$ the space all  functions $u$ such that   such that $\cL u$ is constant in a neighbourhood of $\overline{B}$.\footnote{For the convenience of the reader we observe that in \cite{MMV3} these spaces were denoted by $h_1^2(B), q^2_1(B), h_1^2(\overline{B}), q^2_1(\overline{B})$ respectively, while  in \cite{MMV2} the space $q^2(\overline{B})$ was denoted by $Q^1_B$. }
%\end{itemize}
\par
We shall refer to $h^2(B)$ as the \emph{harmonic Bergman space
} on $B$, while functions in $q^2(\OV B)$ are referred to as \emph{quasi-harmonic functions on
$\OV B$}. 
 Often we think of $q^2(\OV B)$
as a subspace of $\ld{B}$.  When we do, the symbol $q^2(\OV B)^\perp$
will denote the orthogonal complement of $q^2(\OV B)$ in $\ld{B}$.   Clearly $q^2(B)^\perp$ 
is a subspace of $q^2(\OV{B})^\perp$ and  of $h^2(B)^\perp$.
\begin{definition} \label{def: atoms}
An $X^1$-\emph{atom} associated to
the geodesic ball $B$ is a function $A$ in $\ld{M}$, supported in $B$, such that
\begin{enumerate}
\item[\itemno1]
$\int A\, v\wrt\mu=0$ for all $v\in q^2(\OV B)$;
\item[\itemno2]
$\ds\norm{A}{2}\leq \mu(B)^{-1/2}$.
\end{enumerate}
Note that condition \rmi\ implies that 
$\int_M A\wrt\mu=0$, because $\One_{2B}$ is in $q^2(\overline{B})$.
Given a positive ``scale parameter'' $s$, we say that 
an $X^k$-atom is {\emph{ at scale} $s$} if it is supported 
in a ball $B$ of $\cB_s$. 
\end{definition}

\begin{definition}
Choose a ``scale parameter'' $s>0$.  The \emph{Hardy-type space $\Xu{M}$}
is the space of all functions $F$ that admit a 
decomposition of the form $F= \sum_j c_j\, A_j$, where $\{c_j\}$ is 
a sequence in $\ell^1$ and $\{A_j\}$ is a sequence 
of %admissible 
$X^1$-atoms at scale $s$.  
We endow $\Xu{M}$ with the natural ``atomic norm''
$$
\bignorm{F}{X^1}
:= \Bigl\{ \sum_{j=1}^\infty \bigmod{c_j}:  F= \sum_{j=1}^\infty c_j  A_j,
\hbox{$A_j$  $X^1$-atoms at scale $s$} \Bigr\}. 
$$
\end{definition}

\begin{remark} \label{rem: equivalence Q OVQ}
It is known \cite{MMV1,MMV2} that all these atomic norms are equivalent and
it becomes a matter of convenience to choose one or another.  In our situation
any value $< \Inj(M)$ of the scale parameter $s$ would be 
a convenient choice for the following reasons.  Balls of radius $<\Inj(M)$ have no holes and their boundaries are smooth, 
so that various results concerning Sobolev spaces 
on balls hold.  We shall, implicitly or explicitly, make use
of them in the sequel. 
Another advantage of choosing $s< \Inj(M)$ is that we may make use of the 
fact that the cancellation condition \rmi\ in Definition~\ref{def: atoms} may then be 
equivalently formulated by requiring that 
$A$ be in $q^2(B)^\perp$ \cite[Proposition~3.5 and the comments after Theorem~4.12]{MMV3}.  
This will be used in the sequel without any further comment.
In the following, we shall choose $s_0=\smallfrac{1}{2}\Inj(M)$ and we shall call atoms at scale $s_0$ \emph{admissible}.
\end{remark}
%\noindent
For more on $\Xu{M}$, and on its close generalisations $\Xh{M}$, $k=2,3,\ldots$,
see \cite{MMV1,MMV2, MMV3}.
In particular, it is known that
the spaces $\Xh{M}$ have interesting equivalent characterisations
\cite{MMV2}, that, however, we shall not use in this paper.

\section{Atoms and the Laplace--Beltrami operator}
\label{s: Atoms and the LB}    
Henceforth we denote by $\cL$ the unique self-adjoint extension of minus the Laplace-Beltrami operator on $L^2(M)$. We recall that the domain of $\cL$ is the space of all functions in $L^2(M)$ such that the distribution $\cL u\in L^2(M)$. For a geodesic ball $B$ we denote by $\cL_B$ the restriction of $\cL$ to the subspace
$$
\Dom(\cL_B)=\set{f\in \Dom(\cL): \supp(f)\subset \overline{B}}.
$$
Even though the operator $\cL_B$ is defined on $L^2(M)$, in the following we shall often consider $\cL_B$ as an operator acting on $L^2(B)$.
In addition to $\cL_B$,
we consider also the Dirichlet Laplacian~$\cL_{B,\mathrm{Dir}}$ on the ball $B$, i.e. the Friedrichs extension of the restriction of $\cL$ to $C^\infty_c(B)$. We recall that the domain of $\cL_{B,\mathrm{Dir}}$ is
$$
\Dom(\cL_{B,\mathrm{Dir}})=\set{u\in W^{1,2}_0(B): \cL u\in L^2(B)},
$$
where $\cL u$ is interpreted in the sense of distributions on $B$ and $W^{1,2}_0(B)$ denotes the closure of $C^\infty_c(B)$ in the Sobolev space
$$
W^{1,2}(B)=\set{u\in L^2(B): \mod{\nabla u}\in L^2(B)}
$$
We shall restrict our attention to balls $B$,
which are the interior of their closure and $\partial B$ is smooth.  
Observe that any ball $B$ of radius $<\Inj(M)$ is the interior of its
closure and has smooth boundary. 
The following proposition will be useful later.

\begin{proposition} \label{p: L and LB}
Assume that $B$ is a ball in $M$ with smooth boundary. 
The following hold:
\begin{enumerate}
\item[\itemno1]
$\cL_{B,\mathrm{Dir}}$ is an extension of $\cL_B$;
\item[\itemno2]
$\Ran(\cL_B)=h^2(B)^\perp$ and $\cL_B$ is an isomorphism between its domain, endowed with the graph norm, and its range. 
\item[\itemno3]
%there exists a constant $C$, which does not depend on the ball $B$,
%such that
$$
\bignorm{\cL^{-1}f}{2}
\leq \frac{1}{\la_1(B)} \, \bignorm{f}{\ld{B}} 
\quant f \in h^2(B)^\perp,
$$
where $\la_1(B)$ denotes the first eigenvalue of the Dirichlet Laplacian $\cL_{B,\mathrm{Dir}}$.  
\end{enumerate}
\end{proposition}

\begin{proof}
If $u\in \Dom(\cL_B)$ then $\cL u\in L^2(M)$ and $\supp(u)\subset \overline{B}$. Hence, by elliptic regularity, $u, \mod{\nabla u}\in L_{\mathrm{loc}}^2(M)$. Thus $u\in W^{1,2}(B)$. Since $u=0$ on the complement of $\overline{B}$ and the boundary of $B$ is smooth, the trace of $u$ on the boundary of $B$ is zero. Hence $u\in W^{1,2}_0(B)$ by a classical result.
This proves that $\Dom(\cL_B)\subset \Dom(\cL_{B,\mathrm{Dir}})$. Thus $\cL_B\subset \cL_{B,\mathrm{Dir}}$ because both operators are defined in the sense of distributions on their domains.

Next we prove \rmii.  First we observe that, since functions in 
$\Ran(\cL_B)$ are supported in $\overline{B}$, we may identify isometrically $\Ran(\cL_B)$ with the subspace of $L^2(B)$ obtained by restricting functions to $B$. Thus $\Ran(\cL_B)$ is closed in $L^2(B)$, since it is closed  in $L^2(M)$, because $\cL$ is strictly positive and closed. 
Thus, to prove the inclusion $h^2(B)^\perp\subseteq \Ran(\cL_B)$, it suffices to show that $\Ran(\cL_B)^\perp\subseteq h^2(B)$. Now, if $g\in L^2(B)$ is orthogonal to $\Ran(\cL_B)$, then 
$$
0=\int_B \cL \psi\,\overline{g}\wrt\mu=\langle\psi,{\cL g}\rangle \quant \psi \in C_c^\infty(B),
$$
where $\cL g$ is in the sense of distributions on $B$.  Therefore $\cL g = 0$ in $B$,
i.e., $g$ is harmonic in $B$ and belongs to $\ld{B}$, i.e., $g\in h^2(B)$.
\par
To prove the opposite inclusion, we observe that by \cite[Prop. 3.5]{MMV3} 
$$
h^2(B) = \OV{h^2(\OV B)}.
$$ 
Thus, to prove the inclusion $\Ran(\cL_B)\subseteq h^2(B)^\perp$  it suffices to show that $\Ran(\cL_B)$ is orthogonal
to $h^2(\OV B)$, i.e. that 
$\int_B\cL_B f\, \overline{g}\wrt\mu = 0$ for all $f$  in 
$\Dom(\cL_B)$ and for all $g$ in $h^2(\overline{B})$. 
Pick $f\in \Dom(\cL_B)$, $g\in h^2(\overline{B})$ and
denote by $\hat{g}$ an extension of $g$ to all of $M$, which is in $
\Dom(\cL)$.  Since $\cL_B f=\cL f$ and $\supp(\cL f)
\subset \overline{B}$,
\begin{align*}
\int_B \cL_B f\,\overline{g}\wrt\mu&=\int_M \cL f\, \overline{\hat{g}}\wrt\mu 
=\int_M f\, \overline{\cL \hat{g}}\wrt\mu=0,
%\int_B f\, \overline{\cL \hat{g}}\wrt\mu=0,
\end{align*}
because $\supp(f)\subseteq \overline{B}$ and $\cL \hat{g}$ vanishes in a neighbourhood of $\overline{B}$. 
This concludes the proof that $\Ran(\cL_B)=h^2(B)^\perp$. \par
Next, we observe that the operator $\cL_B$ is injective and continuous from its domain, endowed with the graph norm,  and its range, since it is the restriction of $\cL$ which is injective and closed. Thus the fact that $\cL_B$ is an isomorphism between its domain and its range follows from the Open Mapping Theorem, since  the range $h^2(B)^\perp$ is closed.

Finally, to prove \rmiii, we observe that by \rmii \ if $f\in h^2(B)^\perp$ then there exists $u\in \Dom(\cL_B)$ such that $f=\cL_B u=\cL u$. Thus $\cL^{-1}f=u=\cL_B^{-1}f=\cL_{B,\mathrm{Dir}}^{-1}f$, since $\cL_{B,\mathrm{Dir}}^{-1}$ is an extension of $\cL_B^{-1}$, by \rmi. Hence
$$
\norm{\cL^{-1}u}{2}=\norm{\cL_{B,\mathrm{Dir}}^{-1}f}{2}\le \frac{1}{\lambda_1(B)}\norm{f}{2},
$$
as required.
\end{proof}

\begin{remark} \label{rem: 12h atomi}
Note that if $A$ is an $X^1$-atom supported in $B$, then 
the function $\cL^{-1}A$ has support contained in $\OV{B}$ \cite[Remark~3.5]{MMV2}.
\end{remark}

\noindent
A straightforward consequence of Proposition~\ref{p: L and LB} is the following.

\begin{corollary} \label{c: action on atoms}
If $A$ is an $X^{1}$-atom with support contained in  $\overline{B}$ and $r_B<\Inj(M)$ then the support of $\cL^{-1} A$ is contained in $\overline{B}$ and
\begin{equation} \label{f: action on atoms}
 \bignorm{\cL^{-1}A}{2}
\leq \frac{1}{\la_1(B) \, \mu(B)^{1/2}}.
\end{equation}
\end{corollary}

\begin{proof}
The proof of Proposition~\ref{p: L and LB} (or Remark~\ref{rem: 12h atomi} above)
shows that the support of $\cL^{-1}A$
is contained in $\OV B$.  The estimate (\ref{f: action on atoms}) is
a direct consequence of the size estimate in the definition of an atom 
and of the norm estimate for $\cL^{-1}$ in Proposition~\ref{p: L and LB}~\rmiii.
\end{proof}

\noindent
This result sheds light on the definition of $(1,2,M)$-atom
in \cite{HLMMY}.  In fact, a direct consequence of (\ref{f: action on atoms})
is that if $A$ is an $X^1$-atom
and $\la_1(B) \asymp r_B^{-2}$, then $A$ is an $(1,2,M)$-atom for every
positive integer $M$.  A similar observations applies to $X^k$-atoms
for $k\geq 2$.  
This suggests that the normalisation of $(1,2,M)$-atoms
introduced in \cite{HLMMY} may be profitably modified on manifolds
whenever the geometry of $M$ determines a somewhat different behaviour
of $\la_1(B)$.

\section{Boundedness of imaginary powers}
\label{s: IP}

In this section we analyse the boundedness of $\cL^{iu}$ 
from $X^1(M)$ to $\lu{M}$ in the case where $M$ satisfies our standing assumptions.
In this case the (minimal) heat kernel $h_t$ of $M$ satisfies the following 
pointwise estimate:
\begin{equation} \label{f: heat upper estimate}
h_t(x,y) 
\leq \frac{C}{\min(1,t^{n/2})} \,\,  \e^{-{b} t - d(x,y)^2/(2Dt)}  
\quant x,y \in M \quant t>0.
\end{equation}
See, for instance, \cite{Gr1}.  In particular 
under our standing assumptions, $M$ possesses the following Faber--Krahn inequality
\begin{equation} \label{f: FK}
\la_1(\Om) \geq a \, \mu(\Om)^{-2/n},
\end{equation}
where $a$ is a positive constant and $\Om$ is any precompact region in $M$.

We recall the following special case of \textit{Takeda's inequality}, 
which holds on all connected, complete, noncompact Riemannian manifolds
(see, for instance, \cite[Theorem~12.9]{Gr2}).
Suppose that $B$ is a ball in $M$.  Then 
\begin{equation}
\int_B \bigl(\cH_t \One_{(2B)^c}\bigr)^2 \wrt \mu
\leq \e\, \mu\bigl((2B)\setminus B\bigr) \, 
\bignorm{\cH_t \One_{(2B)^c}}{\infty}^{\!\!\!\!\! 2} \, \, \,
\max\,\Bigl(\frac{r_B^2}{2t}, \frac{2t}{r_B^2}\Bigr) \, \e^{-r_B^2/(2t)}
\end{equation}
for all $t>0$.  Observe that $\cH_t$ is submarkovian, so that 
$$
\bignorm{\cH_t \One_{(2B)^c}}{\infty}
\leq 1 
\quant t > 0.
$$
Under our standing assumptions on $M$, for each $s>0$
there exist constants $C_1$ and~$C_2$ such that
$$
C_1 \mu(B) \leq \mu\bigl((2B)\setminus B\bigr) 
\leq C_2 \mu(B) 
\quant B \in \cB_s.
$$
Then, by Takeda's inequality and the estimate above, there exist positive 
constants~$c$ and $C$ such that 
\begin{equation} \label{f: conseq Takeda}
\frac{1}{\mu(B)} \, \int_B \bigl(\cH_t \One_{(2B)^c}\bigr)^2 \wrt \mu
\leq C\, \, \e^{-c r_B^2/t}
\quant t \in  (0, r_B^2]   \quant B \in \cB_s.  
\end{equation}
%Suppose also that the 
%following parabolic Harnack inequality holds
%$$
%\sup_{(1/2)B} 
%\bigl(\cH_t \One_{(2B)^c}\bigr)^2   
%\leq C \, \frac{1}{\mu(B)} \, \int_B \bigl(\cH_t \One_{(2B)^c}\bigr)^2 \wrt \mu
%\quant B \in \cB_m.
%$$
%\begin{equation} \label{f: conseq Takeda}
%\sup_{(1/2)B} 
%\bigl(\cH_t \One_{(2B)^c}\bigr)^2   
%\leq C\, \, \e^{-c r_B^2/t}
%\quant t >0 \quant B \in \cB_m.
%\end{equation}
%Furthermore, by 

\begin{theorem} \label{t: IP}
Suppose that $M$ is a Riemannian manifold satisfying our standing assumptions.
Then for every $u$ in $\BR\setminus \{0\}$
the imaginary powers $\cL^{iu}$ are bounded from $X^1(M)$ to $\lu{M}$. 
\end{theorem}

\begin{proof}
In view of the theory developed in \cite{MMV3} it suffices to prove that
$$
\sup\, \bigl\{ \bignorm{\cL^{iu}A}{1}: \, \hbox{$A$ admissible $X^1$-atom} \bigr\}
< \infty,   
$$ 

Recall that admissible $X^1$-atoms are supported in balls of radius at most $s_0=\smallfrac{1}{2}\Inj(M)$.  
Suppose that $A$ is such an atom, with support contained in $B$.  
Observe that 
$$
\bignorm{\cL^{iu}A}{1} 
 = \bignorm{\One_{2B} \, \cL^{iu}A}{1} + \bignorm{\One_{(2B)^c} \, \cL^{iu}A}{1}.
$$
We estimate the two summands on the right hand side separately. 
To estimate the first, simply observe that, by Schwarz's inequality, 
the size condition for $A$, and the spectral theorem,   
$$
\begin{aligned}
\bignorm{\One_{2B} \, \cL^{iu}A}{1} 
& \leq \mu(2B)^{1/2} \, \bigopnorm{\cL^{iu}}{2} \, \bignorm{A}{2} \\
& \leq \Bigl(\frac{\mu(2B)}{\mu(B)}\Bigr)^{1/2}.  
\end{aligned}
$$
The right hand side is bounded independently of~$B$, because
$\mu$ is locally doubling.
  
To estimate the second summand, 
we denote by $k_{\cL^{iu+1}}(x,y)$ the kernel of the operator $\cL^{iu+1}$. Then, by 
Schwarz's inequality
and (\ref{f: action on atoms}), we obtain
$$
\begin{aligned}
\bignorm{\One_{(2B)^c} \, \cL^{iu}A}{1} 
&  \leq \bignorm{\cL^{-1}A}{2} \, \, \Bigl[\int_B \!\! {\wrt \mu(y)} 
    \Bigl(\int_{(2B)^c} \bigmod{k_{\cL^{iu+1}}(x,y)} \wrt \mu(x) \Bigr)^2 \Bigr]^{1/2} \\  
& \leq \frac{C}{\la_1(B)} \, \, \Bigl[\frac{1}{\mu(B)}\, \int_B \!\! {\wrt \mu(y)} 
    \Bigl(\int_{(2B)^c} \bigmod{k_{\cL^{iu+1}}(x,y)} \wrt \mu(x) \Bigr)^2 \Bigr]^{1/2}.
\end{aligned}
$$
%$$
%\begin{aligned}
%& =    \bignorm{\One_{2B} \, \cL^{iu}A}{1} + \bignorm{\One_{(2B)^c} \, \cL^{iu}A}{1} \\
%& \leq \mu(2B)^{1/2} \, \bigopnorm{\cL^{iu}}{2} \, \bignorm{A}{2} 
%    + \int_{(2B)^c} {\wrt \mu(x)} \Bigmod{\int_{B} k_{\cL^{iu+1}}(x,y) \, 
%       (\cL^{-1}A) (y) \wrt\mu(y) } \\
%& \leq \Bigl(\frac{\mu(2B)}{\mu(B)}\Bigr)^{1/2} 
%    + \frac{1}{\la_1(B)}\,  \sup_{y \in B} \, \int_{(2B)^c}  \bigmod{k_{\cL^{iu+1}}(x,y)} 
%        \wrt\mu(x)  \\
%\end{aligned}
%$$
It remains to show that 
\begin{equation} \label{f: claim I} 
\Bigl[\frac{1}{\mu(B)}\, \int_B \!\! {\wrt \mu(y)} 
   \Bigl(\int_{(2B)^c} \bigmod{k_{\cL^{iu+1}}(x,y)} \wrt \mu(x) \Bigr)^2 \Bigr]^{1/2} 
\leq  C \, {\la_1(B)}, 
\end{equation}  
where $C$ is independent of $B$ in $\cB_{s_0}$.  
Observe that off the diagonal the following formula for the kernel of $\cL^{iu}$ 
holds 
$$
k_{\cL^{iu+1}}(x,y)
= c_u \ioty t^{-iu-1} \, h_t(x,y) \, {\dtt t}.
$$
We write the integral on the right hand side as the sum 
of the integrals over $(0, r_B^2]$ and $(r_B^2,\infty)$.
Note that 
$$
\begin{aligned}
 \int_{(2B)^c}  \Bigmod{\int_{r_B^2}^\infty t^{-iu-1} \, h_t(x,y) \, {\dtt t}} 
       \wrt\mu(x) 
& \leq   {\int_{r_B^2}^\infty {\frac{\dest t}{t^2}} \, 
    \int_{(2B)^c} h_t(x,y) \wrt\mu(x) } \\
& \leq r_B^{-2},
\end{aligned}
$$
because the heat semigroup is contractive on $\ly{M}$.  Hence
\begin{equation} \label{f: int1infty}
\Bigl[\frac{1}{\mu(B)}\, \int_B \!\! {\wrt \mu(y)} 
   \Bigl(\int_{(2B)^c} \Bigmod{\int_{r_B^2}^\infty t^{-iu-1} \, h_t(x,y) \, {\dtt t}} 
\wrt \mu(x) \Bigr)^2 \Bigr]^{1/2} 
\leq  C \, {\la_1(B)},    
\end{equation}
for $r_B^{-2} \leq C\, \la_1(B)$ (just take $\Omega=B$ in formula \eqref{f: FK} above).

We now prove that there exists a constant $C$, independent of $B$, such that 
\begin{equation} \label{f: int01}  
\Bigl[\frac{1}{\mu(B)}\, \int_B \!\! {\wrt \mu(y)} 
   \Bigl(\int_{(2B)^c} 
\Bigmod{\int_0^{r_B^2} t^{-iu-1} \, h_t(x,y) \, {\dtt t}} 
\wrt \mu(x) \Bigr)^2 \Bigr]^{1/2}  
\leq C \, \la_1(B).   
\end{equation}
%$$ 
%\leq C \int_0^1  \frac{t^{-1}}{\min(1,t^{n/2})} \,\,  
%    \e^{-\la_1(M) t - d(x,y)^2/(2Dt)} \, {\dtt t} 
%$$  
%for all $x$ in $(2B)^c$ and for all $y$ in $B$.  We split the
%last integral as the sum of the integrals over the intervals 
%$(0,1]$ and $[1,\infty)$, and estimate them separately.  
%Clearly, the integral over $(0,1)$ is dominated by
By the generalised Minkowski inequality, the left hand side in (\ref{f: int01}) is majorised by
$$
\int_0^{r_B^2} {\frac{\dest t}{t^2}}\, \, \Bigl[\frac{1}{\mu(B)}\, \int_B \!\! {\wrt \mu(y)} 
\, \, \Bigl(\int_{(2B)^c} \, h_t(x,y)  \wrt \mu(x) \Bigr)^2 \Bigr]^{1/2},
$$
which, by (\ref{f: conseq Takeda}), is in turn bounded above by
$$
\begin{aligned}
\int_0^{r_B^2} \e^{-cr_B^2/(2t)} {\frac{\dest t}{t^2}}
& =     \frac{1}{r_B^2} \int_0^{1} \e^{-c/(2v)} {\frac{\dest v}{v^2}} \\
& \leq  C\, r_B^{-2}.
\end{aligned}
$$
%The integral over $[1,\infty)$ is dominated by
%$$
%\begin{aligned}
%&  \int_1^\infty t^{-1} \, \e^{-\la_1(M) t-d(x,y)^2/(2Dt)} \, {\dtt t} \\
%\end{aligned}
%$$
%We may assume that $r_B < 1/2$.  Now, for $y$ in $B$  
%\begin{equation} \label{f: int01}
%\begin{aligned}
% \int_{(2B)^c} \Bigmod{\int_0^1 t^{-iu-1} \, h_t(x,y) \, {\dtt t}}
%       \wrt\mu(x) 
%& \leq  \int_{B^c} \int_0^1 h_t(x,c_B) \, {\frac{\dest t}{t^2}} \wrt\mu(x) \\
%& \leq C\, \int_{B^c}  \frac{1+\e^{-c\, d(x,c_B)^2}}{d(x,c_B)^{2+n}} \wrt\mu(x). 
%\end{aligned}
%\end{equation}
%We split the last integral as a sum of integrals over suitable annuli.  
%It is straightforward
%to check that the last integral is bounded from above by $C \, r_B^{-2}$,
%where $C$ is independent of $B$.  
Finally, note that $r_B^{-2} \leq C\, \la_1(B)$,
and (\ref{f: int01}) is proved.  
Then (\ref{f: int1infty}) and (\ref{f: int01}) prove (\ref{f: claim I}), as
required to  conclude the proof of the theorem.
\end{proof}
\section{Boundedness of the Riesz transform}
\label{s: RT exponential case}

In this section we prove that the Riesz transform is bounded from $\Xu{M}$
to $\lu{M}$. As a preliminary
step, we prove the following:

\begin{lemma} \label{l: est exp}
For every $\eta$ in $(0,1)$ and every $s>0$ there exist positive constants~$c$ and $C$ such that 
for every $B$ in $\cB_s$
\begin{equation} \label{f: est exp}
\int_{(4B)^c} \e^{-d(x,y)^2/Dt} \wrt \mu(x)
\leq C \, \bigl(t^{n/2} \, \e^{-\eta r_B^2/Dt} + \e^{-c/t}\bigr)
\end{equation}
for every $t$ in $(0,r_B^2]$ and for every $y$ in $B$.
\end{lemma}

\begin{proof}
For simplicity we prove the lemma for $s=1$. The general case requires only minor modifications. Since $y\in B$ and $x\notin 4B$,
$$
\begin{aligned}
d(x,y)
& \geq d(x,c_B) - d(y,c_B) \\
& \geq d(x,c_B) - r_B \\
& \geq \frac{1}{2} \, d(x,c_B).
\end{aligned}
$$
Hence
$$
\int_{(4B)^c} \e^{-d(x,y)^2/Dt} \wrt \mu(x)
\leq \int_{(4B)^c} \e^{-d(x,c_B)^2/4Dt} \wrt \mu(x).
$$
Thus, it suffices to estimate the last integral. 
We split the set $(4B)^c$ into annuli.  If $r_B$ is in $(1/4,1]$, then 
we simply write 
$$
(4B)^c
= \bigcup_{k=1}^\infty A\bigl(4kr_B, 4(k+1)r_B\bigr),
$$
where $A(u,v)$ denotes the annulus $\{x \in M: u\leq d(x,c_B)\leq v\}$. 
If, instead, $r_B < 1/4$, then we write
$$
(4B)^c
= \Bigl[\bigcup_{j=0}^{J-1} A\bigl(2^j4r_B, 2^{j+1}4r_B\bigr)\Bigr]
\cup \Bigl[ \bigcup_{k=1}^\infty A\bigl(2^J4kr_B, 2^J4(k+1)r_B\bigr) \Bigr],
$$
where $J$ is chosen so that $R:= 2^{J}4r_B$ is in $(1/2,1]$, i.e.,
$$
\log_2(1/r_B) - 3 \leq J \leq \log_2(1/r_B) - 2.
$$
We give details in the case where $r_B < 1/4$.  
The case where $r_B$ is in $(1/4,1]$ is simpler and we omit the details. 
By (\ref{f: lower bound balls}),
$$  
\begin{aligned}
\int_{A(2^j4r_B, 2^{j+1}4r_B)} \e^{-d(x,c_B)^2/4Dt} \wrt \mu(x)
& \leq C\, (2^{j+1}4r_B)^n \,  \e^{-2^{2j+2}r_B^2/Dt} \\ 
& =    C' \,t^{n/2}\,  \Bigl(\frac{2^{2j+2}r_B^2}{Dt}\Bigr)^{n/2} 
        \,  \e^{-2^{2j+2}r_B^2/Dt} \\ 
& \leq C_\eta\,t^{n/2}\,  \e^{-\eta 2^{2j+2}r_B^2/Dt}.  
\end{aligned}
$$
We have used the fact that $t\leq r_B^2$ in the last inequality. 
By summing over $j$ between $0$ and $J-1$, we obtain that
\begin{equation} \label{f: terms first sum}
\begin{aligned}
\int_{(RB) \setminus (4B)} \e^{-d(x,c_B)^2/4Dt} \wrt \mu(x)
& \leq C_\eta\,t^{n/2}\,  \sum_{j=0}^\infty 
     \bigl[\e^{-4\eta r_B^2/Dt}\bigr]^{2^{2j}} \\ 
& \leq C_\eta\,t^{n/2}\,\e^{-4\eta r_B^2/Dt}.
\end{aligned}
\end{equation}
By (\ref{f: volume growth}) and the estimate 
$(Rk)^\al \,  \e^{2\be R(k+1)} \leq C_\vep \, \e^{(2\be+\vep) Rk}$, which holds
for every~$k$,
\begin{equation} \label{f: est big annuli}
\begin{aligned}
\int_{A(2^J4kr_B, 2^J4(k+1)r_B)} \e^{-d(x,c_B)^2/4Dt} \wrt \mu(x)
& \leq C\, (Rk)^\al \,  \e^{2\be R(k+1)-R^2k^2/4Dt} \\ 
& \leq C_\vep\,  \e^{(2\be+\vep) Rk-R^2k^2/4Dt}.
\end{aligned}
\end{equation}
By completing the square, and using the fact that $t\leq r_B^2$, we see that 
$$    
\begin{aligned}
(2\be+\vep) Rk-\frac{R^2k^2}{4Dt}
& =    \Bigl(\be+\frac{\vep}{2}\Bigr)^2 \, 4Dt - \Bigl[\frac{Rk}{2\sqrt{Dt}}
         - 2\Bigl(\be+\frac{\vep}{2}\Bigr)\, \sqrt{Dt}\Bigr]^2 \\ 
& \leq \Bigl(\be+\frac{\vep}{2}\Bigr)^2 \, 4Dr_B^2 - \Bigl[\frac{Rk}{2\sqrt{Dt}}
         - 2\Bigl(\be+\frac{\vep}{2}\Bigr)\, \sqrt{Dt}\Bigr]^2.  
\end{aligned}
$$
Now observe that if $Rk \geq 4D(2\be+\vep)\, r_B^2$, then 
$Rk - (2\be+\vep)2Dt \geq Rk/2$, so that 
\begin{equation} \label{f: est exponent}
(2\be+\vep) Rk-\frac{R^2k^2}{4Dt}
\leq C- \frac{R^2k^2}{16\, Dt},  
\end{equation}
where $C= \bigl(\be+\vep/2\bigr)^2 \, 4D$.  Choose 
$K:= [\!\![4D(2\be+\vep)\, r_B^2/R]\!\!]+1$.  Now,
$$
\int_{M \setminus (RB)} \e^{-d(x,c_B)^2/4Dt} \wrt \mu(x)
 =  \sum_{k=1}^\infty \int_{A\bigl(2^J4kr_B, 2^J4(k+1)r_B\bigr)} 
       \e^{-d(x,c_B)^2/4Dt} \wrt \mu(x).  
$$
Note that $K \leq D(\be+\vep/2)$, so it does not depend on $r_B$.
We estimate each of the terms of the series up to the $(K-1)^{\mathrm{th}}$ 
as in (\ref{f: est big annuli}), so that the sum for $k$ from $1$ to $K-1$
may be estimated by 
$$
C_\vep\, K\,  \e^{(2\be+\vep)D}\, \e^{-R^2/4Dt}
\leq C\, \e^{-1/8Dt}.
$$
The series for $k$ from $K$ to $\infty$ may be estimates as 
$$
C\, \sum_{k=K}^\infty \e^{-R^2k^2/(16\, Dt)}
\leq C \, \e^{-c/t}
$$
for some positive $c$. By combining the estimates above, we obtain that 
\begin{equation} \label{f: est second int}
\int_{M \setminus (RB)} \e^{-d(x,c_B)^2/4Dt} \wrt \mu(x)
\leq C\, \e^{-c/t},
\end{equation} 
which, together with (\ref{f: terms first sum}), gives the required estimate.  

The proof of the lemma is complete. 
\end{proof}

\begin{lemma}  \label{l: local-global}
Suppose that $M$ is a Riemannian manifold satisfying our standing assumptions.
Fix a scale parameter $s< \Inj(M)$.
Then there exists a constant $C$ such that for every ball $B$ in $\cB_s$
$$
\norm{\nabla\cL^{1/2}f}{\lu{(4B)^c}} 
\leq C\,  r_B^{-2}  \, \bignorm{f}{\lu{B}} 
\quant f \in \lu{B}.
$$
\end{lemma}

\begin{proof}   
\textit{Step~I: reduction of the problem and conclusion}.
A straightforward argument shows that 
$$
\nabla\cL^{1/2} f (x)
= \int_M k_{\nabla\cL^{1/2}}(x,y) \, f(y) \wrt \mu(y)
\quant f \in C_c(M) \quant x \notin \supp (f),
$$
where
\begin{equation} \label{f: kernel Riesz mod}
k_{\nabla\cL^{1/2}}(x,y) = \frac{1}{\Ga(-1/2)} \, \ioty \nabla_xh_t(x,y) \,
     \frac{\wrt t}{t^{3/2}}
\end{equation}
for all $(x,y)$ off the diagonal in $M\times M$.  Here $h_t$ denotes
the heat kernel (with respect to the Riemannian measure $\mu$).  
Define $\cI^B (y)$ and $\cI_B (y)$ by
$$
\cI^B (y)
:= \int_0^{r_B^2} \frac{\wrt t}{t^{3/2}} \, \int_{(4B)^c} \bigmod{\nabla_x h_t(x,y)}
   \wrt \mu(x)
$$
and 
$$
\cI_B (y)
:= \int_{r_B^2}^\infty \frac{\wrt t}{t^{3/2}} 
    \, \int_{(4B)^c} \bigmod{\nabla_x h_t(x,y)} \wrt \mu(x).
$$
Note that, by (\ref{f: kernel Riesz mod}) and Tonelli's theorem,
\begin{equation} \label{f: ref est}
\begin{aligned}
\norm{\nabla\cL^{1/2}f}{\lu{(4B)^c}} 
& \leq \int_{(4B)^c}\ioty \int_B \bigmod{\nabla_x h_t(x,y)} \, \mod{f(y)} 
   \wrt\mu(x) \wrt\mu(y) \frac{\wrt t}{t^{3/2}} \\
& = \int_B \bigl[\cI^B (y)+\cI_B (y)\bigr]  \, \mod{f(y)} \wrt\mu(y).   
\end{aligned}
\end{equation}
We \textit{claim} that there exists a constant $C$ such that
\begin{equation} \label{f: IB}
\cI^B (y) \leq C\, r_B^{-2}
\qquad\hbox{and} \qquad
\cI_B (y) \leq C\, r_B^{-2}.
\end{equation}
These estimates, hence the claim, will be proved 
in Step~II and Step~III, respectively.  Assuming the claim,
we may deduce from \eqref{f: ref est} and \eqref{f: IB} that 
\begin{equation*} %\label{f: conclusion}
\begin{aligned}  
\norm{\nabla\cL^{1/2}f}{\lu{(4B)^c}} 
& \leq \int_B \bigl[\cI^B (y)+\cI_B (y)\bigr]  \, \mod{f(y)} \wrt\mu(y) \\   
& \leq C\,  r_B^{-2}  \, \bignorm{f}{\lu{B}}, 
\end{aligned}
\end{equation*}
as required to conclude the proof of the lemma.  
%In the last inequality, 
%we have used the fact that $r_B \leq s$.  

\textit{Step~II: estimate of $\cI^B (y)$}.  
We shall use Grigor'yan's integral estimates for the gradient of the heat kernel
\cite{Gr3}.  It will be convenient 
to introduce more notation.  We fix $D>4$, and set,
for every $y$ in $M$ and for every $t>0$,
\begin{equation} \label{f: E0}
E_0(y,t)
:= \int_M h_t(x,y)^2 \, \e^{d(x,y)^2/Dt} \wrt \mu(x)
\end{equation}
and
\begin{equation} \label{f: E1}
E_1(y,t)   
:= \int_M \bigmod{\nabla_xh_t(x,y)}^2 \, \e^{d(x,y)^2/Dt} \wrt \mu(x).
\end{equation}
%Furthermore, 
%$E_j(y,t;B)$, $j=0,1$,
%will be defined much as $E_j(y,t)$, but the corresponding integral
%in its definition will be on $(4B)^c$ instead of $M$.  
Recall that, under our standing assumptions
on $M$, the Faber--Krahn type inequality \eqref{f: FK} holds on $M$.  
Furthermore, the constant $a$ in \eqref{f: FK} is uniformly bounded
from below as long as $r_B\leq s$ (because $M$ has bounded geometry).  Therefore
\cite[Theorem~15.8, p.~400]{Gr2}
$$      
E_0(y,t)
\leq C \, t^{-n/2} 
\quant t \in (0,r_B^2] \quant y \in M.
$$
Hence \cite[Theorem~1.1]{Gr3}  
$$    
E_1(y,t)
\leq C \, t^{-n/2-1} 
\quant t \in (0,r_B^2] \quant y \in M.
$$
By using Schwarz's inequality, the estimate above
and Lemma~\ref{l: est exp}, we obtain
\begin{equation} \label{f: est cIB Step III}
\begin{aligned}
\cI^B (y)
& \leq C \, \int_0^{r_B^2} \bigl(t^{n/2} \, \e^{-\eta r_B^2/Dt} 
     + \e^{-c/t}\bigr)^{1/2} \, E_1(y,t)^{1/2} \frac{\wrt t}{t^{3/2}}  \\
& \leq C \, \int_0^{r_B^2} t^{-1} \, \e^{-\eta r_B^2/2Dt} \,\frac{\wrt t}{t} 
     + C \, \int_0^{r_B^2} \e^{-c/2t} \, \frac{\wrt t}{t^{n/4+2}}  \\
& \leq C \, \bigl(r_B^{-2}+1\bigr) 
\quant y \in M,  
\end{aligned}
\end{equation} 
as required to prove the first statement in \eqref{f: IB}.  

\textit{Step~III: estimate of $\cI_B(y)$}.  
The main idea is to combine Caccioppoli's inequality with Harnack's
inequality for balls of small radius.  We denote by $\{\vp_j\}$ a 
smooth partition of unity associated to a locally finite covering 
$\{B_j'\}$ of $(4B)^c$ by balls of radius $r_B$.   
We set  
\begin{equation} \label{f: def cJBjk}
\cI_{B;j,k}(y)
:= \int_{(k-1)r_B^2}^{kr_B^2}\frac{\wrt t}{t^{3/2}} 
     \, \int_{B_j'} \bigmod{\nabla_x h_t(x,y)} \, \vp_j(x) \wrt \mu(x).
\end{equation}
Clearly 
\begin{equation} \label{f: first est for cIBb}
\begin{aligned}
\cI_B(y) 
& \leq \sum_j \int_{r_B^2}^\infty \frac{\wrt t}{t^{3/2}} 
           \, \int_{B_j'} \bigmod{\nabla_x h_t(x,y)} \wrt \mu(x) \\
& =    \sum_j \sum_{k=2}^\infty \cI_{B;j,k}(y).
\end{aligned}
\end{equation}   
We now introduce the parabolic cylinder $\cC_{j,k}$, defined as follows
$$
\cC_{j,k}
:= B_j'\times \bigl((k-1)r_B^2,kr_B^2\bigr].
$$
Clearly $\mu\times \la \bigl(\cC_{j,k}\bigr) = \mu(B_j') \, r_B^2$, where
$\la$ denotes the Lebesgue measure on the real line. 
Recall the following version of the parabolic Caccioppoli inequality
\begin{equation} \label{f: caccioppoli}   
\int_{\cC_{j,k}} \bigmod{\nabla_x h_t(x,y)}^2 \wrt \mu(x) \wrt t
\leq \frac{C}{r_B^2} \, 
\int_{2\cC_{j,k}} \bigmod{h_t(x,y)}^2 \wrt \mu(x) \wrt t,
\end{equation}
where
$$
2\cC_{j,k}
:= 2B_j'\times \bigl((k-2)r_B^2,(k+1) r_B^2\bigr].
$$
This inequality is a straightforward consequence of 
\cite[Lemma~15.2 and Lemma 15.3]{Gr2}.  
Observe that 
$$
\cI_{B;j,k}(y)
\asymp \frac{1}{(kr_B^2)^{3/2}} \, \int_{\cC_{j,k}} \bigmod{\nabla_x h_t(x,y)} 
   \wrt \mu(x) \wrt t.
$$
Therefore, by Schwarz's inequality and Caccioppoli's inequality
$$
\begin{aligned}
\cI_{B;j,k}(y)
& \leq \frac{\mu\times \la \bigl(\cC_{j,k}\bigr)}{(kr_B^2)^{3/2}} \,  
   \Bigl[ \frac{1}{\mu\times \la \bigl(\cC_{j,k}\bigr)} 
    \int_{\cC_{j,k}} \bigmod{\nabla_x h_t(x,y)}^2 \wrt \mu(x) \wrt t\Bigr]^{1/2} \\ 
& \leq \frac{\mu\times \la \bigl(\cC_{j,k}\bigr)}{(kr_B^2)^{3/2}} \,  \frac{1}{r_B}\,
   \Bigl[ \frac{1}{\mu\times \la \bigl(2\cC_{j,k}\bigr)} 
    \int_{2\cC_{j,k}} h_t(x,y)^2 \wrt \mu(x) \wrt t\Bigr]^{1/2}.
\end{aligned}
$$
We now use the parabolic Harnack inequality applied to the parabolic
cylinder $2\cC_{j,k}$ and conclude that 
\begin{align}  \label{f: upH}
%& \qquad 
\Bigl[ \frac{1}{\mu\times \la \bigl(2\cC_{j,k}\bigr)} 
   \int_{2\cC_{j,k}}  h_t(x,y)^2 & \wrt \mu(x) \wrt t\Bigr]^{1/2} 
  \nonumber \\
& \leq C \, \inf_{(z,t) \in 2\cC_{j,k+2}} h_t(z,y) 
\nonumber\\
& \leq C \,  \frac{1}{\mu\times \la \bigl(2\cC_{j,k}\bigr)} 
  \int_{2\cC_{j,k+2}} h_t(x,y) \wrt \mu(x) \wrt t.  
\end{align}
By combining the last two estimates, we obtain that
\begin{equation}\label{f: est1 cIB Step IV}
\begin{aligned}
\cI_{B;j,k}(y)
& \leq \frac{C}{(kr_B^2)^{3/2}} \, \frac{1}{r_B} \, 
     \int_{2\cC_{j,k+2}} h_t(x,y) \wrt \mu(x) \wrt t \\
& \leq \frac{C}{r_B} \, 
     \int_{kr_B^2}^{(k+3)r_B^2} \frac{\wrt t}{t^{3/2}} \, 
      \int_{2B_j'} h_t(x,y) \wrt \mu(x).
\end{aligned}
\end{equation}
We now sum over $j$ and $k$, and then use the facts that the covering $\{B_j'\}$
is uniformly locally finite and that
$\norm{h_t(\cdot,y)}{1}\leq 1$ for every $y$ in $M$, and obtain 
\begin{equation} \label{f: est cIB Step IV}
\begin{aligned}
\cI_B (y)
& \leq \frac{C}{r_B} \int_{r_B^2}^\infty \frac{\wrt t}{t^{3/2}} \, 
      \int_{(2B)^c} h_t(x,y) \wrt \mu(x) \\
& \leq \frac{C}{r_B} \int_{r_B^2}^\infty \frac{\wrt t}{t^{3/2}}  \\
& \leq \frac{C}{r_B^2},   
\end{aligned}
\end{equation} 
as required to prove the second estimate in \eqref{f: IB},
and to conclude the proof of the claim.  
\end{proof}

\begin{theorem}
Suppose that $M$ is a Riemannian manifold satisfying our standing assumptions.
The Riesz transform $\nabla\cL^{-1/2}$ is bounded from $\Xu{M}$ to $\lu{M}$.
\end{theorem}

\begin{proof}
In view of the theory developed in \cite{MMV3},  it
suffices to prove that 
\begin{equation} \label{f: aim}
\sup \bignorm{\nabla\cL^{-1/2} A}{1} 
< \infty,
\end{equation}
where the supremum is taken over all admissible $X^1$-atoms $A$,  i.e. over all atoms at scale {$s_0$}.  

Fix  such an atom \ $A$, and denote by $B$ the ball associated to $A$.  Recall
that $r_B \leq s_0$.  Observe that
$$
\bignorm{\nabla\cL^{-1/2} A}{1} 
= \bignorm{\nabla\cL^{-1/2} A}{\lu{4B}} + \bignorm{\nabla\cL^{-1/2} A}{\lu{(4B)^c}}.
$$
We shall estimate the two summands on the right hand side separately. 
Clearly 
$$
\begin{aligned}
\bignorm{\nabla\cL^{-1/2} A}{\lu{4B}} 
& \leq \mu\bigl(4B\bigr)^{1/2} \, \bignorm{\nabla\cL^{-1/2} A}{\ld{4B}} \\
& \leq \Bigl(\frac{\mu\bigl(4B\bigr)}{ \mu(B)}\Bigr)^{1/2} \\
& \leq C.
\end{aligned}
$$
 In the second inequality 
above we have used the fact that 
$$
\bignorm{\nabla\cL^{-1/2} A}{\ld{4B}}
\leq \norm{A}{2} \leq \mu(B)^{-1/2},   
$$
which follows from the $L^2$-boundedness of the Riesz transform and the 
size property of $A$.  In the last inequality we have used the fact that the measure $\mu$ is locally doubling.
Therefore
\begin{equation} \label{f: est first summand}
\sup \bignorm{\nabla\cL^{-1/2} A}{\lu{4B}} 
< \infty,
\end{equation}
where the supremum is taken over all admissible $X^1$-atoms $A$. 

Thus, to conclude the proof of the theorem it suffices to show that
\begin{equation} \label{f: est second summand}
\sup \bignorm{\nabla\cL^{-1/2} A}{\lu{(4B)^c}} 
< \infty,
\end{equation}
where the supremum is taken over all admissible $X^1$-atoms $A$. 
Observe that 
$$
\nabla \cL^{-1/2}A 
= \nabla \cL^{-1/2}\cL \cL^{-1}A 
= \nabla \cL^{1/2} \bigl(\cL^{-1}A\bigr).
$$
Recall that  by Corollary \ref{c: action on atoms},
$$
%\begin{aligned}
\bignorm{\cL^{-1}A}{\ld{B}} \le   \frac{1}{\la_1(B)}
\ \mu(B)^{-1/2},
%& \leq \bigopnorm{\cL^{-1}}{\ld{B}} \, \bignorm{A}{\ld{B}} \\
%& \leq \bigopnorm{\cL^{-1}}{\ld{B}} \, \mu(B)^{-1/2}, 
%\end{aligned}
$$
so that 
\begin{equation}
\begin{aligned}\label{f: 0.7}
\bignorm{\cL^{-1}A}{\lu{B}}
& \leq \mu(B)^{1/2} \, \bignorm{\cL^{-1}A}{\ld{B}} \\
%& \leq  \bigopnorm{\cL^{-1}}{\ld{B}} \\
& \leq  \frac{1}{\la_1(B)}.
\end{aligned}
\end{equation}
%Here we have used Proposition~\ref{p: L and LB} and the fact that $r_B \leq s$.
Therefore
\begin{equation*} %\label{f: conclusion}
\begin{aligned}  
\norm{\nabla\cL^{-1/2}A}{\lu{(4B)^c}} 
& =    \norm{\nabla\cL^{1/2}\big(\cL^{-1}A\big)}{\lu{(4B)^c}} \\
& \leq C\,  r_B^{-2} \, \bignorm{\cL^{-1}A}{\lu{B}} \\   
& \leq C\,  r_B^{-2}  \, \la_1(B)^{-1} \\   
& \leq C;
\end{aligned}
\end{equation*}
the first inequality follows from Lemma~\ref{l: local-global}, 
the second from (\ref{f: 0.7}), 
and the last from \eqref{f: FK}. 
The proof of the theorem is complete.
\end{proof}

\section{Volume doubling manifolds satisfying Gaussian estimates }
\label{s: RT polynomial case}

The methods developed in 
Sections~\ref{s: IP} and \ref{s: RT exponential case} may be easily adapted to the case where
the manifold $M$ satisfies the following assumptions:
\begin{enumerate}  
\item[\itemno1]
$M$ possesses the \textit{volume doubling property}, i.e., 
there exists a positive constant $D_{\infty}$ such that 
\begin{equation*} \label{f: dou}
\mu(2B)\leq D_{\infty} \, \mu(B)
\quant B \in \cB;   
\end{equation*}
\item[\itemno2]
the heat kernel satisfies a \textit{Gaussian upper estimate}, i.e. there exist positive constants $c, C$ such that 
$$
h_t(x,y)\le \ C\ \frac{1}{\mu(B(y,\sqrt{t}))} \ \e^{-c\frac{d^2(x,y)}{t}}
$$
for all $x,y\in M$ and all $t>0$.
\end{enumerate}
\par
Note that, under the assumptions above on $M$, T.~Coulhon
and X.T.~Duong \cite{CD} proved that the Riesz transform is of weak type $(1,1)$.
The Marcinkiewicz interpolation argument, together with the 
trivial $L^2$ bound for the Riesz transform imply, for every $p$ in $(1,2)$,
the estimate
\begin{equation}  \label{f: Lp est RT}
\bignorm{{\nabla \cL^{-1/2}f}}{p}
\leq C_p \, \norm{f}{p}
\quant f \in \lp{M}.
\end{equation}
Let  $\Xu{M}$ be the space defined much as in the case of manifolds of exponential
growth, but allowing $X^1$-atoms  associated to balls of any positive radius.
We refer the reader to \cite{S} for all basic
properties of $\Xu{M}$.

\begin{theorem}\label{t: d+FK}
Suppose that $M$ is a Riemannian manifold satisfying the volume doubling property and the Gaussian upper estimate.
%relative Faber-Krahn inequality. 
Then the  imaginary powers $\cL^{iu}$, $u\in \BR$, and the Riesz transform $\nabla\cL^{-1/2}$,  are bounded from $\Xu{M}$ to $\lu{M}$.
\end{theorem}

The proof of Theorem \ref{t: d+FK} is an adaptation of the  arguments described in the previous sections. The main modifications are 
\begin{itemize}
\item[\rmi] the replacement of the Faber-Krahn inequality (\ref{f: FK}) with the  \textit{relative Faber-Krahn inequality}: there exist positive constants $b$ and $\nu$ such that 
\begin{equation*} \label{f: RFK}
\la_1(U) \geq \frac{b}{r_B^2} \, \left(\frac{\mu(B)}{\mu(U)}\right)^{2/\nu}
\end{equation*}
for every ball $B$ in $\cB$ and for every relatively compact open set $U\subset B$.\par
It is well known that manifolds that possess the volume doubling property satisfy the relative Faber--Krahn inequality if and only if  the heat kernel satisfies a {Gaussian upper estimate} \cite{Gr1}.
\item[\rmii] The replacement of the uniform parabolic Harnack inequality  in the proof of inequality (\ref{f: upH}) by the following reverse H\"older inequality
for subsolutions
of the heat equation: there exists a constant $C$ such that for all integer $j$ and $k$ with $k\geq 2$
\begin{equation*} %\label{f: claim LiWang}
\int_{2\cC_{j,k}} h_t(x,y)^2 \wrt \mu(x) \wrt t 
\leq \frac{C}{\mu\times \la \bigl(4\cC_{j,k}\bigr)} 
\Bigl[\int_{4\cC_{j,k}} h_t(x,y) \wrt \mu(x) \wrt t \Bigr]^2.
\end{equation*}
To the best of our knowledge, this inequality is due to 
P.~Li and J.~Wang (see the proof of \cite[Theorem~2.1, p.~1269--1270]{LW}).
%, who, in 
%turn, elaborate on a previous result of Li and R.~Schoen~\cite{LS}. 

\end{itemize}

By combining Theorem \ref{t: d+FK} with the interpolation result in \cite{S}, one obtains  \eqref{f: Lp est RT}.  
Thus, we give a different proof of one of 
the main results obtained by Coulhon and Duong.

The result of Theorem \ref{t: d+FK} is not new. Indeed, it can be shown using  the results of \cite{HLMMY}  that if the manifold $M$ is doubling and the heat kernel satisfies a Gaussian upper estimate, then the space $X^1(M)$ coincides with the subspace of $0$-forms in the space $H^1(T^*\Lambda)$ introduced by P. Auscher, A. McIntosh, and E. Russ in \cite{AMR}. Hence the boundedness of the Riesz transform from $\Xu{M}$ to $\lu{M}$ follows from \cite[Theorem 5.13]{AMR} and that of the imaginary powers from \cite[Corollary 4.3]{DY}. However, we believe that the proofs outlined here might be of some interest for their simplicity.

\end{document}